\pdfoutput=1
\documentclass[11pt,a4paper]{amsart}

\usepackage[margin=1in]{geometry}
\usepackage[utf8]{inputenc}
\usepackage[T1]{fontenc}
\usepackage{lmodern,microtype,mathtools,amssymb,amsthm,fontawesome5,thmtools}
\makeatletter
\@ifundefined{newcounteralias}{}{%
  \renewcommand\thmt@autorefsetup{%
    \@xa\def\csname\thmt@envname autorefname\@xa\endcsname
      \@xa{\thmt@thmname}%
  }%
}
\makeatother

\usepackage{environ}
\usepackage[noadjust,compress]{cite}
\usepackage[dvipsnames]{xcolor}
\usepackage[colorlinks,citecolor=brown]{hyperref}
\usepackage[nameinlink,capitalize,noabbrev]{cleveref}
\usepackage{todonotes}
\usepackage{orcidlink}

\newcommand{\prz}[2][]{}
\newcommand{\jack}[2][]{}
\newcommand{\james}[2][]{}

\makeatletter
\newcounter{restate}
\NewEnviron{savethm}[2]{%
  \global\expandafter\let\csname restate@#2\endcsname\BODY
  \begin{#1}\label{#2}\BODY\end{#1}%
}
\newcommand{\restate}[2]{%
  \global\advance\c@restate\@ne
  \begingroup
  \@namedef{the#1}{\ref*{#2}}%
  \@namedef{theH#1}{restate.\arabic{restate}}%
  \let\label\@gobble
  \begin{#1}[Restatement]%
    \csname restate@#2\endcsname
  \end{#1}%
  \endgroup
}
\makeatother

\AtBeginDocument{\belowdisplayshortskip\belowdisplayskip}

\newcommand{\emaillink}[1]{\NoCaseChange{\href{mailto:#1}{\raisebox{0.05em}{\textcolor{gray}{\fontsize{9}{12}\sffamily\bfseries\faIcon[regular]{envelope}}}}}}

\newcommand{\defn}[1]{\textcolor{RoyalBlue}{#1}}

\newcommand{\NN}{\mathbb{N}}
\newcommand{\RR}{\mathbb{R}}
\newcommand{\Oh}{\mathcal{O}}
\newcommand{\calC}{\mathcal{C}}

\newcommand{\calP}{\mathcal{P}}
\newcommand{\calS}{\mathcal{S}}
\newcommand{\calU}{\mathcal{U}}
\renewcommand{\Pr}{\mathsf{Pr}}
\newcommand{\dbar}{\smash[t]{\overline d}}
\newcommand{\Gbar}{\smash[t]{\overline G}}
\newcommand{\Qbar}{\smash[t]{\overline Q}}
\newcommand{\Rbar}{\smash[t]{\overline R}}
\newcommand{\dotpartial}{\smash[t]{\dot\partial}}
\newcommand{\Ksubdiv}[1]{\smash[t]{K_{#1}^{(1)}}}
\newcommand{\Ksubdivnosmash}[1]{K_{#1}^{(1)}}

\let\leq\leqslant
\let\geq\geqslant
\let\setminus\smallsetminus
\let\epsilon\varepsilon

\newtheorem{theorem}{Theorem}[section]

\newtheorem{corollary}[theorem]{Corollary}
\newtheorem{conjecture}[theorem]{Conjecture}
\newtheorem{claim}[theorem]{Claim}

\newtheorem{lemma}[theorem]{Lemma}
\newtheorem{observation}[theorem]{Observation}
\newtheorem{proposition}[theorem]{Proposition}

\crefname{conjecture}{Conjecture}{Conjectures}
\crefname{observation}{Observation}{Observations}
\crefname{appendixsection}{Appendix}{Appendices}

\title[Revisiting Separators in String Graphs]{Revisiting Separators in String Graphs, Region Intersection Graphs, and Induced-Minor-Free Graphs}

\author{James Davies}
\address{Leipzig University, Germany}
\email{jgdavies@uwaterloo.ca}
\author{Paweł Rzążewski}
\address{Warsaw University of Technology, Poland}
\email{pawel.rzazewski@pw.edu.pl}
\author{Jack Spalding-Jamieson}
\address{Independent}
\email{jacksj@uwaterloo.ca}
\author{Bartosz Walczak}
\address{Department of Theoretical Computer Science, Faculty of Mathematics and Information Science, Jagiellonian University, Kraków, Poland}
\email{bartosz.walczak@uj.edu.pl}

\thanks{James Davies was supported by the Alexander von Humboldt Foundation in the framework of the Alexander von Humboldt Professorship of Daniel Kráľ endowed by the Federal Ministry of Education and Research.
Paweł Rzążewski was supported by the National Science Centre grant 2024/54/E/ST6/00094.
Bartosz Walczak was partially supported by the National Science Centre grant 2019/34/E/ST6/00443.}

\newcommand{\camerareadyonly}[1]{}

\begin{document}

\begin{abstract}
String graphs are intersection graphs of continuous curves in the plane.
In 2010, Fox and Pach proved that string graphs with $m$ edges admit balanced vertex separators of size $\mathcal{O}(m^{3/4}\sqrt{\log m})$ and conjectured that this bound can be improved to $\mathcal{O}(\sqrt{m})$, which would be optimal.
After partial progress, Lee claimed in 2017 to prove this conjecture, and more generally to prove analogous separator theorems for region intersection graphs over minor-closed graph families.
A key intermediate step in Lee's approach was a padded decomposition scheme for region intersection graphs.
In 2024, Korhonen and Lokshtanov extended Lee's approach to the broader setting of graph classes excluding a fixed graph as an induced minor.
These results have since found applications in a wide range of works.

Later in 2024, a flaw was discovered in the construction of Lee's padded decomposition scheme, invalidating Lee's as well as Korhonen and Lokshtanov's separator results, and many works relying on them.

In this work, we prove that string graphs with $m$ edges admit balanced separators of size $\mathcal{O}(\sqrt{m})$, and more generally obtain analogous $\mathcal{O}(\sqrt{m})$ bounds for region intersection graphs over fixed minor-closed graph families and for graphs excluding a fixed graph as an induced minor, thereby repairing the vast majority of the affected works.
Furthermore, our balanced separators can be found in polynomial time.

Our proof uses a different approach from Lee's and does not rely on padded decompositions.
Since padded decompositions have also found applications independent of the separator theorems, particularly for string graphs, we also prove a padded decomposition result for region intersection graphs that is optimal for string graphs.

We have formalized both the balanced separator and the padded decomposition results in Lean~4 with mathlib.

\end{abstract}

\maketitle

\section{Introduction}

In 1979, Lipton and Tarjan~\cite{lipton1979separator} showed that every $n$-vertex planar graph admits a \emph{balanced separator}~-- a set whose removal breaks the graph into small pieces~-- of size $\Oh(\sqrt{n})$.
This bound is optimal (up to constant factors) since the $\sqrt{n}\times\sqrt{n}$ grid is planar and does not admit a balanced separator of size $o(\sqrt{n})$.

Since the seminal work of Lipton and Tarjan, balanced separators have been a central object of study in graph theory and theoretical computer science, with applications in many different areas, most notably algorithm design~\cite{lipton1980applications,baker1994approximation,robertson1995disjoint,demaine2005subexponential,demaine2008bidimensionality,klein2013recursive}.
In particular, sublinear balanced separators have been shown to exist for various graph classes.
For example, graphs of genus $g$ admit balanced separators of size $\Oh(\sqrt{gn})$~\cite{gilbert1984genus} and graphs that exclude $K_h$ as a minor admit balanced separators of size $\Oh\bigl(\min\bigl\{h\cdot\sqrt{n}+f(h),\:h\sqrt{\log h}\cdot\sqrt{n}\bigr\}\bigr)$, for a fast-growing function $f$~\cite{alon1990separator,kawarabayashi2010separator,le2026separator}.
Note that the latter result also applies to all classes excluding an $h$-vertex graph as a minor.

Another way of generalizing the separator result is to look at classes of graphs defined by intersecting geometric objects.
For a family $\calS$ of objects, typically compact and connected subsets of $\RR^2$, their \emph{intersection graph} has a vertex for every object in $\calS$, and an edge exists if and only if the corresponding objects intersect.
For example, Miller, Teng, Thurston, and Vavasis~\cite{miller1997sphere} showed that disk graphs (i.e., intersection graphs of disks in the plane) admit balanced separators of size $\Oh(\sqrt{\omega n})$, where $\omega$ is the maximum size of a clique in the graph.
We remark that the dependence on $\omega$ cannot be dropped, since the complete graph $K_n$ is a unit disk graph, and does not admit a balanced separator of size $o(n)$.
The above result of Miller et al.\ was later extended to more general intersection graphs of fat, similarly-sized objects~\cite{smith1998geometric}, and to sparse intersection graphs of related objects~\cite{fox2026separators}.

What about intersection graphs of arbitrary connected subsets of the plane, also known as \emph{string graphs}?
Here we cannot hope for a balanced separator of size $o(n)\cdot f(\omega)$, for any function $f$, since the complete bipartite graph $K_{n,n}$ is a string graph and has no sublinear separators.
However, in 2010, Fox and Pach~\cite{fox2010separator} proved that string graphs with $m$ \emph{edges} admit balanced separators of size $\Oh(m^{3/4}\sqrt{\log m})$.
They also conjectured that this bound can be improved to $\Oh(\sqrt{m})$.

\begin{conjecture}[Fox and Pach~\cite{fox2010separator}]
\label{conj:stringsep}
Every\/ $m$-edge string graph admits a balanced separator of size\/ $\Oh(\sqrt{m})$.
\end{conjecture}

Note that this would be optimal since every planar graph is a string graph and its number of edges is linear in the number of vertices.
In 2013, Matoušek~\cite{matousek2014near} improved the separator size to $\Oh(\sqrt{m}\log m)$, almost closing the gap, and in 2017, Lee~\cite{lee2017separators} claimed to finally prove \cref{conj:stringsep}.
Actually, Lee did not work with string graphs, but with more general \emph{region intersection graphs}.
We say that a graph $G$ is a region intersection graph over a graph $H$ if vertices of $G$ can be mapped to connected subsets of $H$ so that $G$ is the intersection graph of these subsets.
Lee claimed the following result.

\begin{theorem}\label{thm:mainrig}
For every\/ $h\geq 1$ there exists a constant\/ $c_h$ with the following property.
Every\/ $m$-edge region intersection graph over a\/ $K_h$-minor-free graph admits a balanced separator of size at most\/ $c_h\cdot\sqrt{m}$.
\end{theorem}

Lee's argument yields $c_h=\Oh(h^3\sqrt{\log h})$.
As string graphs are exactly region intersection graphs over planar graphs~\cite{kratochvil1986string}, which are $K_5$-minor-free, \cref{thm:mainrig} immediately implies \cref{conj:stringsep}.

In 2024, Korhonen and Lokshtanov extended \cref{thm:mainrig} to the even broader setting of \emph{induced-minor-free} graph classes~\cite{korhonen2024induced}.
We say that a graph $H$ is an induced minor of a graph $G$ if $H$ can be obtained from $G$ by a sequence of vertex deletions and edge contractions.
A graph is \emph{$H$-induced-minor-free} if it does not contain $H$ as an induced minor.
Let $\Ksubdiv{h}$ denote the $1$-subdivision of $K_h$, bisecting every edge exactly once.
By directly using some intermediate results of Lee as a black box, Korhonen and Lokshtanov claimed a form of the following result.

\begin{theorem}
\label{thm:mainim}
Let\/ $G=(V,E)$ be an\/ $m$-edge graph excluding\/ $\Ksubdiv{h}$ as an induced minor.
Then\/ $G$ admits a balanced separator of size
\[\Oh\left(\min\bigl(\log m,h^2\bigr)\cdot h\cdot\sqrt{m}\right){.}\]
Furthermore, there is a deterministic polynomial-time algorithm that, given a graph\/ $G$, either obtains such a balanced separator or concludes that\/ $G$ contains\/ $\Ksubdiv{h}$ as an induced minor.
In the latter case, there is a second Las Vegas randomized polynomial-time algorithm that produces an induced minor model of\/ $\Ksubdiv{h}$ in\/ $G$.
\end{theorem}

We note that if $G$ is a region intersection graph over a $K_h$-minor-free graph, then $G$ is $\Ksubdiv{h}$-induced-minor-free.
On the other hand, for every $h$, there is a $\smash[t]{K_6^{(1)}}$-induced-minor-free graph that is not a region intersection graph over any $K_h$-minor-free graph~\cite{bonnet2025approximation}.
Thus, \cref{thm:mainim} is a strict generalization of \cref{thm:mainrig}.

Over the years, many different papers have built on these results, as well as the different intermediate results~\cite{bonnet2019optimality,okrasa2020subexponential,novotna2021subexponential,kumar2021constant,kumar2022point,lokshtanov2024approximation,bonnet2023treewidth,berthe2025pushing,spaldingjamieson2025reweighted}.
Unfortunately, in 2024, a significant error was discovered in a core component of Lee's original proof of \cref{thm:mainrig}~\cite{bonnet2025approximation}, invalidating or weakening key results in the follow-up work.
(Following~\cite{bonnet2025approximation}, the inequality at the sixth line of the proof of Lemma~4.14 of the full version of Lee's paper~\cite{lee2017separators} need not hold.)

In this paper, we present a new proof of \cref{thm:mainim} that does not rely on the flawed component of Lee's proof.
As previously noted, our result implies \cref{thm:mainrig} and thus proves \cref{conj:stringsep}.
As a consequence, we also fix the majority of papers that were broken by the invalidation of Lee's result.

Our proof uses two standard ideas via a multicommodity flow framework shared by much of the previous work~\cite{matousek2014near,lee2017separators,korhonen2024induced}.
First, we apply a multicommodity flow bound proven by Korhonen and Lokshtanov~\cite{korhonen2024induced} to show that the graph admits a multicommodity flow of value $\Omega(\sqrt{m})$, whose dual form gives us a vertex-weighted metric over the graph.
Second, we construct a metric embedding into the line.
It is in the details of this second idea that our techniques diverge somewhat from Lee's methods, so the intermediate results we obtain are not all identical.
In \cref{sec:conclusion} we list results relying on the work of Lee and discuss whether and how they are fixed by our results.

\subsection*{Padded Decompositions and Metric Embeddings}

As mentioned, a key step in the original claimed proofs of both \cref{thm:mainrig} and \cref{thm:mainim} is to use a lemma of Lee's that gives a certain kind of metric embedding into the line.
To this aim, Lee made use of so-called padded decompositions for vertex-weighted shortest-path metrics.
Padded decompositions have seen a number of (effectively equivalent) definitions.
We use the one by Conroy and Filtser~\cite{conroy2025protect}.

Throughout this paper, a \defn{metric} on a set $X$ is a function $d\colon X\times X\to\RR_{\geq 0}$ such that $d(x,x)=0$, $d(x,y)=d(y,x)$, and $d(x,z)\leq d(x,y)+d(y,z)$; so we allow a metric to assign distance zero to two distinct points.%
\footnote{Such a function is often called a \emph{pseudometric} or \emph{semimetric}, but we prefer the simpler term \emph{metric}.}
A ball in a metric space $(X,d)$ is denoted by $\defn{B_d(x,r)}:=\{y\in X:d(x,y)\leq r\}$.
For a metric space $(X,d)$, a partition $\calP$ of $X$ is \defn{$\Delta$-bounded} if every part $P\in\calP$ has diameter at most $\Delta$, i.e., $\max_{x,y\in P}d(x,y)\leq\Delta$.
For a partition $\calP$ of $X$ and a point $x\in X$, let \defn{$\calP(x)$} denote the part containing $x$.
A \defn{$(\beta,\delta,\Delta)$-padded decomposition} for $(X,d)$ is a probability distribution over the $\Delta$-bounded partitions of $X$ such that for every $x\in X$ and every $\gamma\in[0,\delta]$, it holds that $\Pr[B_d(x,\gamma\Delta)\subseteq\calP(x)]\geq e^{-\beta\gamma}$.
The parameter $\beta$ is referred to as its \defn{padding}, and the parameter $\delta$ is referred to as its probability.
A \defn{$(\beta,\delta)$-padded decomposition scheme} for $(X,d)$ is then a family of $(\beta,\delta,\Delta)$-padded decompositions for all values of $\Delta>0$.

Throughout this paper, we implicitly assume that the graph denoted by $G$ is always connected\footnote{This assumption does not affect the statements of our results on the existence of balanced separators.}, which allows us to speak of (finite) shortest-path distances in $G$.
For a graph $G=(V,E)$ and an assignment of vertex weights $w\colon V\to\RR_{\geq 0}$, we define two vertex-weighted shortest-path metrics.
The \defn{endpoint-light vertex-weighted shortest-path metric $d_w^G$} assigns to every path $u_1u_2\cdots u_k$ ($k\geq 2$) length $\frac{1}{2}w(u_1)+w(u_2)+\dots+w(u_{k-1})+\frac{1}{2}w(u_k)$; equivalently, it is the edge-weighted shortest-path metric over $G$ with each edge $uv$ given weight $\frac{1}{2}(w(u)+w(v))$.
The \defn{endpoint-heavy vertex-weighted shortest-path metric $\dbar_w^G$} assigns to every path $u_1u_2\cdots u_k$ ($k\geq 2$) length $w(u_1)+w(u_2)+\dots+w(u_{k-1})+w(u_k)$.
(Both metrics assign length $0$ to paths with no edges.)
We have defined the two metrics for clarity, since prior works use the two variations inconsistently.
We will primarily work with the endpoint-light variation for our new results, but we will apply some previous results that are stated for the endpoint-heavy variation.%
\footnote{There is a so-called $2$-bi-Lipschitz embedding between the two types of metrics, so they are often interchangeable up to a constant factor.}

Lee~\cite[Theorem~2.10]{lee2017separators} claimed the following.
\begin{itemize}
\item[($\star$)] If $G$ is a region intersection graph over a graph $G_0$ that excludes $K_h$ as a minor, then for every assignment of vertex weights $w\colon V\to\RR_{\geq 0}$, the endpoint-light vertex-weighted shortest-path metric $d_w^G$ admits an $\bigl(\Oh(h^2),\Omega(1)\bigr)$-padded decomposition scheme.
\end{itemize}

This particular claim has seen many applications beside \cref{thm:mainrig}.
Unfortunately, this component of Lee's original proof relies on a statement that was incorrect, and a fix for this component remains to be found.
(We will take a different approach.)

The subsequent papers that have applied this component of Lee's can be partitioned into three groups (see also the discussion in \cref{sec:conclusion}):
\begin{enumerate}
\item those using this result only for string graphs;
\item those using this result for its implications in terms of metric embeddings (discussed shortly);
\item those requiring this result in its full generality.
\end{enumerate}
The majority of subsequent papers fall into the first two categories, and we will provide sufficient results for both of those categories.

For the first category, we will prove the following result, which has a worse dependence on $h$.

\begin{savethm}{theorem}{thm:pad-rig}
Let\/ $G=(V,E)$ be a region intersection graph over a\/ $K_h$-minor-free graph\/ $G_0$, and let\/ $w\colon V\to\RR_{\geq 0}$ be an assignment of vertex weights.
Then the endpoint-light and endpoint-heavy vertex-weighted shortest-path metrics\/ $d_w^G$ and\/ $\dbar_w^G$ each admit\/ $\bigl(\Oh(h^2),\Omega(1/h)\bigr)$-padded decomposition schemes.
In particular, every string graph admits an\/ $(\Oh(1),\Omega(1))$-padded decomposition scheme.
\end{savethm}

For the second category, we must first describe the special kinds of metric embeddings that are useful for separators.
In fact, this will be critical for our proof of \cref{thm:mainim}.
Fix a finite metric space $(X,d)$.
An embedding $f\colon X\to\RR$ of $(X,d)$ into the line is said to be \defn{non-expansive} if $|f(x)-f(y)|\leq d(x,y)$ for all $x,y\in X$.
The \defn{average distortion} of a non-expansive embedding is
\[\frac{\sum_{x,y\in X}d(x,y)}{\sum_{x,y\in X}|f(x)-f(y)|}{.}\]

We call non-expansive embeddings into the line with bounded average distortion \defn{Rabinovich embeddings}, since Rabinovich~\cite{rabinovich2008average} was the first to study them directly.
It should be noted that every finite metric space with $n$ points admits a Rabinovich embedding with (average) distortion $\Oh(\log n)$~\cite[Theorem~3.1]{feige2008improved}.

As a corollary of ($\star$), Lee (implicitly) claimed the following result on Rabinovich embeddings.

\begin{theorem}
\label{thm:old-rig-rabinovich}
Let\/ $G=(V,E)$ be a region intersection graph over a\/ $K_h$-minor-free graph\/ $G_0$, and let\/ $w\colon V\to\RR_{\geq 0}$ be an assignment of vertex weights.
Then the endpoint-light vertex-weighted shortest-path metric\/ $d_w^G$ admits a non-expansive embedding into the line with average distortion at most\/ $\Oh(h^2)$.
\end{theorem}

We will prove the following result, that in particular implies \cref{thm:old-rig-rabinovich} as a special case.

\begin{savethm}{theorem}{thm:embed}
Let\/ $G=(V,E)$ be a graph excluding\/ $\Ksubdiv{h}$ as an induced minor, and let\/ $w\colon V\to\RR_{\geq 0}$ be an assignment of vertex weights.
Then the metrics\/ $d_w^G$ and\/ $\dbar_w^G$ admit non-expansive embeddings into the line with average distortion
\[\min\left\{\Oh(\log{|V|}),\:625h^2\right\}{.}\]

Furthermore, there is a polynomial-time in\/ $|V|$ algorithm that, given a graph $G$ and an assignment of vertex weights\/ $w\colon V\to\RR_{\geq 0}$, obtains either an induced minor model of\/ $\Ksubdiv{h}$ in\/ $G$, or such a non-expansive embedding into the line.
\end{savethm}

This will be enough to prove \cref{thm:mainim} and, consequently, \cref{thm:mainrig}, as well as fix a number of prior works building on Lee's.

However, we believe that ($\star$) should hold as well, even in the more general setting of forbidden induced minors, and even for a stronger parameter dependence generalizing an analogous result for $K_h$-minor-free graphs~\cite{conroy2025protect}.
To this end, we pose the following conjecture.

\begin{conjecture}
Let\/ $G$ be a\/ $\Ksubdiv{h}$-induced-minor-free graph, and let\/ $w\colon V\to\RR_{\geq 0}$ be an assignment of vertex weights.
Then the endpoint-light vertex-weighted shortest-path metric\/ $d_w^G$ admits an\/ $\bigl(\Oh(\log h),\Omega(1)\bigr)$-padded decomposition scheme.
\end{conjecture}

\subsection*{Improved Parameter Dependence}

In \cref{app:improved-bound}, we improve the dependence on $h$ in \cref{thm:embed}, and consequently in \cref{thm:mainim}, obtaining the following result.

\begin{savethm}{theorem}{thm:improved-main}
Let\/ $G=(V,E)$ be an\/ $m$-edge graph excluding\/ $\Ksubdiv{h}$ as an induced minor.
Then\/ $G$ admits a balanced separator of size
\[\Oh\bigl(h\log h\cdot\sqrt{m}\bigr){.}\]
Furthermore, there is a deterministic polynomial-time algorithm that, given a graph\/ $G$, either obtains such a balanced separator or concludes that\/ $G$ contains\/ $\Ksubdiv{h}$ as an induced minor.
In the latter case, there is a second Las Vegas randomized polynomial-time algorithm that produces an induced minor model of\/ $\Ksubdiv{h}$ in\/ $G$.
\end{savethm}

\subsection*{Other Related Works}

In late 2025, Davies~\cite{davies2025string} and, independently, Chang, Conroy, Tan, and Zheng~\cite{chang2026cutting} proved that every unweighted string graph admits an embedding into an edge-weighted planar graph that preserves all pairwise shortest-path distances up to a universal constant factor.
We emphasize that these results do not resolve \cref{conj:stringsep}: they apply to unweighted shortest-path metrics, whereas Lee's approach requires padded decompositions for vertex-weighted shortest-path metrics.
In fact, vertex-weighted string graphs do not admit embeddings into planar graphs with universally bounded distortion, as we show in \cref{app:weighted-planar}.

\subsection*{Formalization}

To avoid further errors, we have formalized the existential statements in \cref{thm:mainim}, \cref{thm:pad-rig}, and \cref{thm:improved-main} in Lean~4 with mathlib.
Our formalization of these first two results was almost entirely constructed using the proprietary LLM-based tool Aristotle~\cite{achim2025aristotle} (with some small manual edits, and other small edits using a coding agent), although we have carefully separated the definitions and main theorem statements from the proofs, and carefully hand-verified them so that we can be confident in the result.
The improved bound was formalized using GPT 5.6 Sol.
We note that the formalization differs very slightly in some of its proof details compared to what we present in this paper, and also in its constants.
However, all asymptotic dependencies on $n$, $m$, and $h$ are the same.
Our formalization is publicly available.%
\footnote{\url{https://github.com/jacketsj/formalized-induced-minor-free-separators}}

\section{From Non-Expansive Embeddings to Balanced Separators}
\label{sec:standard-relax-embed-sep-results}

We now outline the techniques we will use to prove \cref{thm:mainim} using \cref{thm:embed}.

Similarly to previous works, we follow the pipeline of Matoušek~\cite{matousek2014near} for proving the existence of separators in a graph class, based on two different relaxations of a sparsest vertex-cut problem in the style of Leighton and Rao~\cite{leighton1999multicommodity}.

For a graph $G=(V,E)$, we define its \defn{extremal shortest-path spread} as
\[Q(G) := \max_{w\colon V\to\RR_{\geq 0}}\enspace\sum_{u,v\in V}\dbar_w^G(u,v)\quad\text{subject to}\quad\sum_{v\in V}w(v)\leq 1{.}\]
Note that extremal-shortest-path spread is using endpoint-heavy shortest paths.
Its \defn{extremal line-embedded spread} is
\begin{align*}
\Qbar(G) := \max_{\substack{w\colon V\to\RR_{\geq 0}\\f\colon V\to\RR}}\enspace &\,\,\sum_{\!\!\!u,v\in V\!\!\!}\,\,|f(u)-f(v)| \\
\text{subject to}\enspace &\,\,\sum_{v\in V}\,\,w(v)\leq 1{,}\quad w(u)+w(v)\geq|f(u)-f(v)|\quad\text{for all }uv\in E{.}
\end{align*}

These two quantities are related by the following result.

\begin{savethm}{proposition}{prop:sp-to-embed}
Let\/ $G=(V,E)$ be a graph such that for every assignment of vertex weights\/ $w\colon V\to\RR_{\geq 0}$, the endpoint-heavy vertex-weighted shortest-path metric\/ $\dbar_w^G$ admits a non-expansive embedding into the line with average distortion at most\/ $\alpha$, for some\/ $\alpha>0$.
Then\/ $\alpha\cdot\Qbar(G)\geq Q(G)$.
\end{savethm}

\Cref{prop:sp-to-embed} is a standard result, and has essentially been proven in many different papers (e.g.~\cite{biswal2010eigenvalue}).
For completeness, we include the proof in~\cref{sec:vertexexpansion}, along with a proof of the following standard result.

\begin{savethm}{proposition}{prop:embed-to-sep}
Let\/ $G=(V,E)$ be a graph such that every connected induced subgraph\/ $H$ of\/ $G$ with\/ $n'$ vertices satisfies\/ $\Qbar(H)\geq(n')^2/\beta$, for some value\/ $\beta>0$.
Then\/ $G$ admits a\/ $\frac{2}{3}$-balanced vertex separator of size at most\/ $4\beta$.
\end{savethm}

With \cref{prop:sp-to-embed} and \cref{prop:embed-to-sep}, a bound on the size of the smallest $\frac{2}{3}$-balanced vertex separator for a graph $G$ will follow from two steps:
\begin{enumerate}
\item a lower bound on $Q(G)$;
\item a proof of the existence of a non-expansive embedding with bounded average distortion.
\end{enumerate}
This is how the original claimed proofs of both \cref{thm:mainrig} and \cref{thm:mainim} operate.
In particular, only the second step was erroneous in each paper.
Korhonen and Lokshtanov correctly proved the following result.%
\footnote{In their original paper, Korhonen and Lokshtanov had a constant of $40$ in place of $45$.
This is a very minor and fixable issue stemming from a proof in their paper~\cite[Lemma~4.4]{korhonen2024induced}, and it was discovered during our formalization process.}

\begin{proposition}[{Contrapositive of~\cite[Lemma 4.4]{korhonen2024induced}}]
\label{prop:kl-bound}
Suppose a graph\/ $G$ does not contain another graph\/ $H$ as an induced minor.
Let\/ $m_G$ be the number of edges of\/ $G$.
Let\/ $n_H$ and\/ $m_H$ be the numbers of vertices and edges of\/ $H$, respectively.
Then
\[Q(G) > \frac{n_G^2}{45\cdot\sqrt{n_H+m_H}\sqrt{m_G}}{.}\]
Furthermore, given a graph\/ $G$ that does not satisfy this inequality, then there is a Las Vegas randomized polynomial time algorithm to produce an induced-minor model of\/ $H$ in\/ $G$.
\end{proposition}

The combination of \cref{prop:kl-bound}, \cref{prop:sp-to-embed}, \cref{prop:embed-to-sep}, and \cref{thm:embed} is enough to prove \cref{thm:mainim}.

\section{Main Structural Result}

Let $(U,d)$ be a metric space.
An \defn{outer measure} on $U$ is a mapping $\mu$ from the subsets of $U$ to non-negative reals such that $\mu(\emptyset)=0$, $X\subseteq Y$ implies $\mu(X)\leq\mu(Y)$, and $\mu(X\cup Y)\leq\mu(X)+\mu(Y)$.
We normally want to take $\mu(X)={|X|}/{|U|}$, but other choices may be reasonable, e.g., $\mu(X)=\chi(G[X])/\chi(G)$ when $U$ is the vertex set of a graph $G$.

The main result of this section is the following, which is a key new ingredient in our proof of \cref{thm:embed}.
(More precisely, the key ingredient is the algorithmic version of \cref{thm:pure-pairs-subdiv}, i.e., \cref{thm:pure-pairs-subdiv-poly} with $\mu(X)={|X|}/{|V|}$.)

\begin{theorem}
\label{thm:pure-pairs-subdiv}
Let\/ $n,m\in\NN$ with\/ $m\geq 1$, $\epsilon,\eta\geq 0$ with\/ $(2n+2)\epsilon+(n+m)\eta<1$, $G=(V,E)$ be a graph with vertex weights\/ $w\colon V\to\RR_{\geq 0}$, $d=d_w^G$ be the endpoint-light vertex-weighted shortest-path metric of\/ $G$, $\mu$ be an outer measure on\/ $V$ with\/ $\mu(V)=1$, $H$ be a graph with\/ $n$ vertices and\/ $m$ edges, and\/ $\delta\geq 0$.
Then at least one of the following holds:
\begin{enumerate}
\item there is\/ $x\in V$ with\/ $\mu(B(x,11\delta))>\eta$;
\item there are\/ $A,B\subseteq V$ with\/ $\mu(A)>\epsilon$, $\mu(B)>\epsilon$, and\/ $d(A,B)>\delta$;
\item $G$ contains a\/ $1$-subdivision of\/ $H$ as an induced minor.
\end{enumerate}
\end{theorem}

\Cref{thm:pure-pairs-subdiv} is inspired by a result of Chudnovsky, Scott, Seymour, and Spirkl~\cite[6.4]{chudnovsky2021pure} towards their proof that graphs excluding the subdivisions of a graph $H$ and their complements as induced subgraphs have the strong Erdős-Hajnal property.
They gave a version for induced subdivisions rather than induced minors, but for graphs rather than our more general associated metrics.
A similar graph version was also proved for pivot-minors~\cite{davies2025pivot}.

Fix a finite metric space $(U,d)$.
For a set $X\subseteq U$ and a number $r\geq 0$, let $\defn{B(X,r)}=\{y\in U\colon d(X,y)\leq r\}$.
For a number $\delta\geq 0$, the graph \defn{$\Gamma_\delta$} has $U$ as the vertex set and $\{xy\colon d(x,y)\leq\delta\}$ as the edge set.
For $\delta\geq 0$ and a set $X\subseteq U$, the set is \defn{$\delta$-connected} if $\Gamma_\delta[X]$ is connected, and the \defn{$\delta$-components} of $X$ are the connected components of $\Gamma_\delta[X]$.
We require the following lemma.

\begin{lemma}\label{lem:pure-pairs-metric}
Let\/ $n\in\NN$, $\epsilon,\eta\geq 0$ with\/ $2n\epsilon+n\eta<1$, $(U,d)$ be a finite metric space, $\mu$ be an outer measure on\/ $U$ with\/ $\mu(U)=1$, and\/ $\rho,\delta\geq 0$.
Then at least one of the following holds:
\begin{enumerate}
\item there is\/ $x\in U$ with\/ $\mu(B(x,\rho))>\eta$;
\item there are\/ $A,B\subseteq U$ with\/ $\mu(A)>\epsilon$, $\mu(B)>\epsilon$, and\/ $d(A,B)>\delta$;
\item there are\/ $\delta$-connected sets\/ $A_1,\ldots,A_n$ at mutual distance greater than\/ $\rho$ with
\[\mu\bigl(B(A_1,\rho+\delta)\cap\cdots\cap B(A_n,\rho+\delta)\setminus B(A_1\cup\cdots\cup A_n,\rho)\bigr) \geq 1-2n\epsilon-n\eta\text{.}\]
\end{enumerate}
\end{lemma}

\begin{proof}
The proof goes by induction on $n$, where the last inequality for $n=0$ reads $\mu(U)\geq 1$, so it holds by assumption.
Suppose that $n\geq 1$, the lemma holds for $n-1$, and neither (1) nor (2) holds, that is, we have the following:
\begin{itemize}
\item[($\neg$1)] there is no $x\in U$ with $\mu(B(x,\rho))>\eta$;
\item[($\neg$2)] there are no $A,B\subseteq U$ with $\mu(A)>\epsilon$, $\mu(B)>\epsilon$, and $d(A,B)>\delta$.
\end{itemize}
It follows that there are $\delta$-connected sets $A_1,\ldots,A_{n-1}$ at mutual distance greater than $\rho$ such that $\mu(C)\geq 1-2(n-1)\epsilon-(n-1)\eta$, where
\[C = B(A_1,\rho+\delta)\cap\cdots\cap B(A_{n-1},\rho+\delta)\setminus B(A_1\cup\cdots\cup A_{n-1},\rho){.}\]
Let $Y=\emptyset$ initially.
Grow a set $X\subseteq C\setminus Y$ starting from an arbitrary point and adding points one by one while keeping $X$ $\delta$-connected, stopping when $\mu(B(X,\rho)\cup Y)>\epsilon$ or when no more points can be added, i.e., the current set $X$ is a $\delta$-component of $C$.
In the latter case, add the current set $X$ to $Y$, set $X=\emptyset$, and continue growing $X\subseteq C\setminus Y$ as before starting from an arbitrary point.
This maintains an invariant that $Y$ is a union of $\delta$-components of $C$.
Since $\mu(C)>\epsilon$ and by ($\neg$1), the process must eventually stop with $\epsilon<\mu(B(X,\rho)\cup Y)\leq\epsilon+\eta$.
Let $Z=C\setminus(B(X,\rho+\delta)\cup Y)$.
Since $d(B(X,\rho)\cup Y,Z)>\delta$, it follows from ($\neg$2) that $\mu(Z)\leq\epsilon$.
Let $A_n=X$.
We conclude that $A_n$ is $\delta$-connected, $d(A_1\cup\cdots\cup A_{n-1},A_n)>\rho$, and
\begin{align*}
&\mu\bigl(B(A_1,\rho+\delta)\cap\cdots\cap B(A_n,\rho+\delta)\setminus B(A_1\cup\cdots\cup A_n,\rho)\bigr) \\
&\qquad= \smash[t]{\mu\bigl(C\cap B(X,\rho+\delta)\setminus B(X,\rho)\bigr) \geq \mu(C)-\mu(B(X,\rho)\cup Y)-\mu(Z)} \\
&\qquad\geq \smash[t]{1-2(n-1)\epsilon-(n-1)\eta-2\epsilon-\eta = 1-2n\epsilon-n\eta}{.}\qedhere
\end{align*}
\end{proof}

Let $G$ be a graph with non-negative edge lengths, $d$ be the shortest-path metric of $G$, and $\delta\geq 0$.
A \defn{$\delta$-fat model} of a graph $H=(V_H,E_H)$ in $G$ is a mapping $\phi$ from $V_H\cup E_H$ to connected vertex subsets of $G$ such that for any two $a,b\in V_H\cup E_H$, if $a$ and $b$ are incident (one being a vertex and the other being an edge), then $\phi(a)\cap\phi(b)\neq\emptyset$, otherwise $d(\phi(a),\phi(b))>\delta$.
Such $\delta$-fat models (or $\delta$-fat minors) have been a major subject of recent research in coarse graph theory; see~\cite{georgakopoulos2025graph} for a survey.

We will derive \cref{thm:pure-pairs-subdiv} from the following (essentially stronger) result for $\delta$-fat models.

\begin{theorem}\label{thm:pure-pairs-metric}
Let\/ $n,m\in\NN$ with\/ $m\geq 1$, $\epsilon,\eta\geq 0$ with\/ $2n\epsilon+(n+m-1)\eta<1$, $G=(V,E)$ be a graph with non-negative edge lengths, $d$ be the shortest-path metric of\/ $G$, $\mu$ be an outer measure on\/ $V$ with\/ $\mu(V)=1$, $H$ be a graph with\/ $n$ vertices and\/ $m$ edges, and\/ $\delta\geq 0$.
Then at least one of the following holds:
\begin{enumerate}
\item there is\/ $x\in V$ with\/ $\mu(B(x,11\delta))>\eta$;
\item there are\/ $A,B\subseteq V$ with\/ $\mu(A)>\epsilon$, $\mu(B)>\epsilon$, and\/ $d(A,B)>\delta$;
\item there is a\/ $\delta$-fat model of\/ $H$ in\/ $G$.
\end{enumerate}
\end{theorem}

\begin{proof}
\Cref{lem:pure-pairs-metric} applied with $U=V$ and $\rho=4\delta$ yields (1), (2), or $\delta$-connected sets $A_1,\ldots,A_n$ at mutual distance greater than $4\delta$ with $\mu(C)\geq 1-2n\epsilon-n\eta$, where
\[C = B(A_1,5\delta)\cap\cdots\cap B(A_n,5\delta)\setminus B(A_1\cup\cdots\cup A_n,4\delta){.}\]
Assume the latter outcome, and suppose that (1) does not hold, that is, there is no $x\in V$ with $\mu(B(x,11\delta))>\eta$.
Since $\mu(C)\geq 1-2n\epsilon-n\eta>(m-1)\eta$, we can greedily find $m$ points $x_1,\ldots,x_m\in C$ at mutual distance greater than $11\delta$.

We construct a $\delta$-fat model $\phi$ of $H$ in $G$ as follows.
Let $v_1,\ldots,v_n$ and $e_1,\ldots,e_m$ denote the vertices and the edges of $H$.
For each $i\in\{1,\ldots,n\}$, let $\phi(v_i)=B(A_i,0.5\delta)$; this is a connected vertex subset of $G$, as $A_i$ is $\delta$-connected.
It follows that $d(\phi(v_i),\phi(v_j))\geq d(A_i,A_j)-\delta>3\delta$ when $i\neq j$.
For each $k\in\{1,\ldots,m\}$, let $i,j\in\{1,\ldots,n\}$ be such that $e_k=v_iv_j$, and let $\phi(e_k)$ be the union of shortest paths connecting $x_k$ to $A_i$ and $A_j$, either of which has length at most $5\delta$.
It is clear that an incident vertex-edge pair is mapped to intersecting vertex subsets.
For any two edges $e_k$ and $e_\ell$ with $k\neq\ell$, since $\phi(e_k)\subseteq B(x_k,5\delta)$, $\phi(e_\ell)\subseteq B(x_\ell,5\delta)$, and $d(x_k,x_\ell)>11\delta$, it follows that $d(\phi(e_k),\phi(e_\ell))>\delta$.
Finally, consider an edge $e_k$ and a vertex $v_\ell$ that are not incident.
Suppose towards a contradiction that $d(\phi(e_k),\phi(v_\ell))\leq\delta$, which implies $d(\phi(e_k),A_\ell)\leq 1.5\delta$.
There is a vertex $u$ on one of the two paths forming $\phi(e_k)$, say, a shortest path from $x_k$ to $A_i$, such that $d(u,A_\ell)\leq 1.5\delta$.
Since $d(x_k,A_\ell)>4\delta$ and $d(A_i,A_\ell)>4\delta$, it follows that $d(x_k,u)>2.5\delta$ and $d(A_i,u)>2.5\delta$, so $d(x_k,A_i)>5\delta$, contradicting the fact that $x_k\in C$.
This shows that $d(\phi(e_k),\phi(v_\ell))>\delta$.
\end{proof}

We are now ready to prove \cref{thm:pure-pairs-subdiv}.

\begin{proof}[Proof of \cref{thm:pure-pairs-subdiv}]
Let $Z=\{x\in V\colon w(x)>\delta\}$.
If $\mu(\{x\})>\eta$ for some $x\in Z$, then (1) holds.
Otherwise, if $\mu(Z)>2\epsilon+\eta$, then (2) holds, as $d(x,y)>\delta$ for any distinct $x,y\in Z$.
For the rest of the proof, suppose $\mu(Z)\leq 2\epsilon+\eta$ and thus $\mu(V\setminus Z)\geq\mu(V)-\mu(Z)\geq 1-2\epsilon-\eta$.
Define an outer measure $\mu'$ on $V\setminus Z$ and constants $\epsilon',\eta'$ by
\[\mu'(X) = \frac{\mu(X)}{\mu(V\setminus Z)}{,} \qquad \epsilon' = \frac{\epsilon}{\mu(V\setminus Z)}{,} \qquad \eta' = \frac{\eta}{\mu(V\setminus Z)}{.}\]
It follows that $\mu'(V\setminus Z)=1$, and since $(2n+2)\epsilon+(n+m)\eta<1$, we have
\[2n\epsilon'+(n+m-1)\eta' = \frac{2n\epsilon+(n+m-1)\eta}{\mu(V\setminus Z)} < \frac{1-2\epsilon-\eta}{1-2\epsilon-\eta} = 1{.}\]

\Cref{thm:pure-pairs-metric} applied with $G-Z$, $\mu'$, $\epsilon'$, $\eta'$ in place of $G$, $\mu$, $\epsilon$, $\eta$ and with $d'$ and $B'$ denoting the metric and the balls defined for $G-Z$, yields at least one of the following outcomes:
\begin{enumerate}
\item[(1$'$)] there is $x\in V\setminus Z$ with $\mu'(B'(x,11\delta))>\eta'$;
\item[(2$'$)] there are $A,B\subseteq V\setminus Z$ with $\mu'(A)>\epsilon'$, $\mu'(B)>\epsilon'$, and $d'(A,B)>\delta$;
\item[(3$'$)] there is a $\delta$-fat model of $H$ in $G-Z$ (with respect to the metric $d'$).
\end{enumerate}
By the definition of $\mu'$, $\epsilon'$, and $\eta'$ we have $\mu'(X)>\epsilon'$ ($\eta'$) if and only if $\mu(X)>\epsilon$ ($\eta$, respectively).
Moreover, for any two sets $A,B\subseteq V\setminus Z$, we have $d'(A,B)\leq\delta$ if and only if $d(A,B)\leq\delta$, because every path in $G$ passing through a vertex in $Z$ has length greater than $\delta$.
So outcomes (1$'$) and (2$'$) yield (1) and (2), respectively.
Now, suppose we have outcome (3$'$).
Modify the model $\phi$ by replacing, for every edge $e=uv$ of $H$, the set $\phi(e)$ by the set of inner vertices of a shortest path in $\phi(e)$ connecting $\phi(u)$ and $\phi(v)$, to represent the vertex subdividing $e$ in the requested model of the $1$-subdivision of $H$ in $G$.
Since $w(x)\leq\delta$ for every $x\in V\setminus Z$, every edge of $G-Z$ has length at most $\delta$, so the $\delta$-fatness property implies that $\phi$ becomes an induced minor model of the $1$-subdivision of $H$ in $G$.
\end{proof}

When $\mu(X)$ can be computed in polynomial time, as is the case when $\mu(X)={|X|}/{|V|}$, the proof of \cref{thm:pure-pairs-subdiv} yields a polynomial time algorithm for finding one of the three outcomes.
Thus, we obtain the following.

\begin{theorem}
\label{thm:pure-pairs-subdiv-poly}
Let\/ $n,m\in\NN$ with\/ $m\geq 1$, $\epsilon,\eta\geq 0$ with\/ $(2n+2)\epsilon+(n+m)\eta<1$, $G=(V,E)$ be a graph with vertex weights\/ $w\colon V\to\RR_{\geq 0}$, $d=d_w^G$ be the endpoint-light vertex-weighted shortest-path metric of\/ $G$, $H$ be a graph with\/ $n$ vertices and\/ $m$ edges, and\/ $\delta\geq 0$.
Then, in time polynomial in\/ $|V|$, one can obtain one of the following outputs:
\begin{enumerate}
\item a vertex\/ $x\in V$ with\/ ${|B(x,11\delta)|}>\eta{|V|}$;
\item vertex sets\/ $A,B\subseteq V$ with\/ ${|A|}>\epsilon{|V|}$, ${|B|}>\epsilon{|V|}$, and\/ $d(A,B)>\delta$;
\item an induced minor model of the\/ $1$-subdivision of\/ $H$ in\/ $G$.
\end{enumerate}
\end{theorem}

\section{Rabinovich Embeddings}

We will now discuss how to construct non-expansive embeddings with bounded average distortion.
These will be our main tool for relating $Q(G)$ and $\Qbar(G)$.

Fix a finite metric space $(X,d)$.
Let $\Delta:=n^{-2}\sum_{x,y\in X}d(x,y)$ be the average pairwise distance.

First, note that for every set $S$, the embedding $f_S\colon X\to\RR$ defined by $f_S(x):=d(x,S)=\min_{s\in S}d(x,s)$ is non-expansive.

We split into two cases.
The first corresponds to the first outcome of \cref{thm:pure-pairs-subdiv-poly}.

\begin{lemma}[Large-core case]
\label{lemma:large-core}
Let\/ $\kappa>0$ and\/ $c>2$.
If there is\/ $x\in X$ with\/ $|B(x,\frac{\Delta}{c})|\geq\kappa n$, then the non-expansive embedding\/ $f_{B(x,\frac{\Delta}{c})}$ has average distortion at most\/ $[\frac{\kappa}{2}(1-\frac{2}{c})]^{-1}$.
\end{lemma}

\begin{proof}
We closely follow the proof by Biswal, Lee, and Rao~\cite[Theorem~4.4, Case~I]{biswal2010eigenvalue}.

Let $S:=B(x,\frac{\Delta}{c})$.
Let $f(x):=f_S(x)=d(x,S)$.
First,
\begin{align*}
n^2\Delta &= \sum_{u,v\in X}d(u,v) \leq \sum_{u,v}\bigl(d(u,x)+d(v,x)\bigr) = 2n\sum_{u\in X}d(u,x) \\
&\leq 2n\sum_{u\in X}\left(d(u,S)+\frac{\Delta}{c}\right) \leq \frac{2n^2\Delta}{c}+2n\sum_{u\in X}d(u,S){.}
\end{align*}
Therefore,
\[n\sum_{u\in X}d(u,S) \geq \frac{1}{2}\left(1-\frac{2}{c}\right)n^2\Delta{.}\]
Next,
\begin{align*}
\sum_{x,y\in X}\lvert f(x)-f(y)\rvert &\geq \sum_{x\in S}\sum_{y\notin S}\lvert f(x)-f(y)\rvert = \sum_{x\in S}\sum_{y\notin S}d(y,S) = {|S|}\sum_{y\notin S}d(y,S) \\
&= {|S|}\sum_{x\in X}d(x,S) \geq \kappa n\sum_{x\in X}d(x,S) \geq \frac{\kappa}{2}\left(1-\frac{2}{c}\right)n^2\Delta \\
&= \frac{\kappa}{2}\left(1-\frac{2}{c}\right)\sum_{x,y\in X}d(x,y){.}\qedhere
\end{align*}
\end{proof}

The second case corresponds to the second outcome of \cref{thm:pure-pairs-subdiv-poly}.

\begin{lemma}[Well-spread case]
\label{lemma:well-spread}
Let\/ $\kappa>0$ and\/ $c>2$.
Let\/ $A,B\subseteq X$ be disjoint sets with\/ $\min\{|A|,|B|\}\geq\kappa n$.
If\/ $d(A,B)\geq\frac\Delta c$, then the non-expansive embedding\/ $f_A$ has average distortion at most\/ $\frac{c}{\kappa^2}$.
\end{lemma}

\begin{proof}
Simply observe that
\begin{align*}
\sum_{x,y\in X}\lvert f(x)-f(y)\rvert &\geq \sum_{x\in A}\sum_{y\in B}\lvert f(x)-f(y)\rvert = \sum_{x\in A}\sum_{y\in B} d(y,A) = {|A|}\sum_{y\in B}d(y,A) \\
&\geq {|A|}\sum_{y\in B}\frac{\Delta}{c} \geq {|A|}\cdot{|B|}\frac{\Delta}{c} \geq \frac{\kappa^2}{c}n^2\Delta = \frac{\kappa^2}{c}\sum_{x,y\in X}d(x,y){.}\qedhere
\end{align*}
\end{proof}

\section{Proofs of Main Results}
\label{sec:main-proofs}

We are now ready to prove \cref{thm:embed}.

\restate{theorem}{thm:embed}

\begin{proof}
Let $n:=|V|$, $d:=d_w^G$, and $\Delta:=n^{-2}\sum_{u,v\in V}d(u,v)$.
If $h\leq 1$, then the graph is empty, so assume $h>1$.
If $\Delta=0$, then it suffices to embed all vertices to the same point, so assume $\Delta>0$.

Since every $n$-point metric admits a non-expansive embedding into the line with average distortion $\Oh(\log n)$~\cite[Theorem~3.1]{feige2008improved}, it suffices to prove a bound of $625h^2$.
Set $\mu(X):=\frac{|X|}{n}$ (for every $X\subseteq V$), $\epsilon:=\frac{13}{40h}$, $\eta:=\frac{1}{32h^2}$, and $\delta:=\frac{\Delta}{33}$.
Note that
\[(2h+2)\epsilon+\left(h+\binom{h}{2}\right)\eta = \frac{213(h+1)}{320h} < 1{.}\]
Apply \cref{thm:pure-pairs-subdiv-poly} with $H=K_h$ to get either some $x\in V$ with $|B(x,\frac{\Delta}{3})|>\eta n$, or disjoint sets $A,B\subseteq V$ with $|A|,|B|>\epsilon n$ and $d(A,B)>\frac{\Delta}{33}$.
By \cref{lemma:large-core} (with $\kappa:=\eta$) and \cref{lemma:well-spread} (with $\kappa:=\epsilon$), the endpoint-light vertex-weighted shortest-path metric $d_w^G$ admits a non-expansive embedding into the line with average distortion at most
\[\max\left\{\frac{6}{\eta},\,\frac{33}{\epsilon^2}\right\} = \max\left\{192h^2,\,\frac{52800}{169}h^2\right\} < 313h^2{.}\]

Since $d_w^G\leq\dbar_w^G\leq 2d_w^G$, the same embedding function is non-expansive for $\dbar_w^G$ and has average distortion at most $\frac{2\cdot 52800}{169}h^2<625h^2$.
\end{proof}

Finally, we can prove our main theorem:

\begin{theorem}[{Extended restatement of \cref{thm:mainim}}]
Let\/ $G=(V,E)$ be an\/ $m$-edge graph excluding\/ $\Ksubdiv{h}$ as an induced minor.
Then\/ $G$ admits a\/ $\frac{2}{3}$-balanced vertex separator of size
\[\min\left\{\Oh(\log{|V|}),\,140000h^2\right\}\cdot h\cdot\sqrt{m}{.}\]
Furthermore, there is a deterministic polynomial-time algorithm that, given a graph\/ $G$, either obtains such a balanced separator or concludes that\/ $G$ contains\/ $\Ksubdiv{h}$ as an induced minor.
In the latter case, there is a second Las Vegas randomized polynomial-time algorithm that produces an induced minor model of\/ $\Ksubdiv{h}$ in\/ $G$.
\end{theorem}

\begin{proof}
Let $n:=|V|$.
If $m=0$, then there is nothing to prove.

Let $F$ be a connected induced subgraph of $G$ with $n'\geq2$ vertices and $m'$ edges, so $F$ is also $\Ksubdiv{h}$-induced-minor-free.
By \cref{thm:embed} and \cref{prop:sp-to-embed},
\[Q(F) \leq \min\left\{\Oh(\log n),\,625h^2\right\}\Qbar(F){.}\]
By \cref{prop:kl-bound}, applied with $H=\Ksubdivnosmash{h}$,
\[Q(F) > \frac{(n')^2}{45\sqrt{|V(\Ksubdivnosmash{h})|+|E(\Ksubdivnosmash{h})|}\sqrt{m'}} = \frac{(n')^2}{45\sqrt{h+3\binom{h}{2}}\sqrt{m'}} = \frac{(n')^2}{45\sqrt{\frac{3}{2}h^2-\frac{1}{2}h}\sqrt{m'}} \geq \frac{(n')^2}{56h\sqrt m}{.}\]
Hence,
\[\Qbar(F) \geq \frac{(n')^2}{\min\left\{\Oh(\log n),\,35000h^2\right\}\cdot h\cdot\sqrt m}{.}\]
Finally, \cref{prop:embed-to-sep} yields the result.
\end{proof}

\section{Padded Decompositions of Region Intersection Graphs}

In this section, we prove \cref{thm:pad-rig}.
Loosely speaking, to do this, our main tools are a recent bound on the asymptotic dimension (defined shortly) of unweighted region intersection graphs, and a new argument to apply this in a black-box manner to bound the Assouad-Nagata dimension (also defined shortly) of vertex-weighted shortest-path metrics in region intersection graphs.
To complete the proof, we also require some existing tools connecting Assouad-Nagata dimension and padded decompositions.

For a metric space $(X,d)$ and sets $A,B\subseteq X$, denote $d(A,B):=\min\{d(a,b)\colon a\in A,\:b\in B\}$.
A family $\calU$ of subsets of $X$ is said to be \defn{$\rho$-disjoint} if $d(U,U')>\rho$ for all distinct $U,U'\in\calU$.
It is said to be \defn{$(k,\rho)$-disjoint} if it can be written as a disjoint union of subfamilies $\calU=\calU_1\cup\cdots\cup\calU_k$, where each $\calU_i$ is $\rho$-disjoint.
A function $f\colon\RR_{>0}\to\RR_{>0}$ is a \defn{$k$-dimensional control function} for $(X,d)$ if, for every $\rho>0$, there is a $(k+1,\rho)$-disjoint partition of $X$ into sets of diameter at most $f(\rho)$.
The \defn{asymptotic dimension} of $(X,d)$ is the minimum $k$ for which there is some $k$-dimensional control function.
The \defn{Assouad-Nagata dimension} of $(X,d)$ is the minimum $k$ for which there is some constant $c>0$ such that $\rho\mapsto c\rho$ is a $k$-dimensional control function.
Such a function is called a \defn{dilation}.

We will also use the notion of a sparse cover.
Like padded decompositions, there are many (equivalent) definitions of sparse covers, so we use the one given by Conroy and Filtser~\cite{conroy2025protect}.
For a metric space $(X,d)$, we say a collection $\calC$ of subsets of $X$ is a \defn{$(\beta,s,\Delta)$-sparse cover} if every set in $\calC$ has diameter at most $\Delta$, every ball of radius $\Delta/\beta$ is fully contained in (``covered by'') some set of $\calC$, and every point of $X$ belongs to at most $s$ sets of $\calC$.
We call $\beta$ the \defn{padding}, and $s$ the \defn{sparsity}.
We say that $(X,d)$ admits a \defn{$(\beta,s)$-sparse cover scheme} if it admits a $(\beta,s,\Delta)$-sparse cover for every $\Delta>0$.

\begin{proposition}
\label{prop:an-to-sparse}
Let\/ $(X,d)$ be a finite metric space.
If\/ $\rho\mapsto c\rho$ is a\/ $k$-dimensional control function for\/ $(X,d)$, then\/ $(X,d)$ admits a\/ $(2c+2,k+1)$-sparse cover scheme.
\end{proposition}

\begin{proof}
Fix some $\Delta>0$, and set $r:=\Delta/(2c+2)$.
Let $\calU=\calU_1\cup\cdots\cup\calU_{k+1}$ be a $(k+1,2r)$-disjoint partition of $X$ into sets of diameter at most $2cr$.
For each $U\in\calU$, let $U^+:=\{x\in X\colon d(x,U)\leq r\}$.
Let $\calC:=\{U^+\colon U\in\calU\}$.
Every set of $\calC$ has diameter at most $2cr+2r=\Delta$.
If $x\in U$, then $B_d(x,r)\subseteq U^+$, so every ball of radius $r=\Delta/(2c+2)$ is contained in some set of $\calC$.
Finally, each point in $X$ belongs to at most one enlarged set coming from each family $\calU_i$, since otherwise two distinct sets of $\calU_i$ would be at distance at most $2r$.
Thus, every point in $X$ belongs to at most $k+1$ sets in $\calC$.
\end{proof}

\begin{proposition}[{\cite[Theorem~4]{conroy2025protect}}]
\label{prop:sparse-to-pad}
Every metric space that admits a\/ $(\beta,s)$-sparse cover scheme also admits an\/ $\bigl(8\beta(1+\ln s),\frac{1}{4\beta}\bigr)$-padded decomposition scheme.
\end{proposition}

We note that Conroy and Filtser stated this proposition in terms of edge-weighted shortest-path metrics, and left the specification of the constants in their proof~\cite[Proof of Theorem~4]{conroy2025protect}.

Recently, Davies~\cite{davies2025string} proved that unweighted string graphs have asymptotic dimension at most $2$.
More generally, Abrishami, Briański, Davies, Du, Masaříková, Rzążewski, and Walczak \cite[Theorem~1.3]{abrishami2026burling} proved that unweighted region intersection graphs over $K_h$-minor-free graphs have asymptotic dimension at most $2^{h-2}-1$.
In the process, they proved the following key result that we will use.

\begin{proposition}[{\cite[Theorem~5.2 of the arXiv version]{abrishami2026burling}}]
\label{prop:rig-control}
For\/ $h\geq 2$, let\/ $G=(V,E)$ be a region intersection graph over a\/ $K_h$-minor-free graph\/ $G_0$.
Let\/ $d^G$ be the unweighted shortest-path metric in\/ $G$ (where the length of an edge is exactly\/ $1$).
Then, for every\/ $\delta\geq 0$, $V$ admits a\/ $(2^{h-2},\delta)$-disjoint partition into sets of radius at most\/ $2(h-1)\delta+1$ in\/ $d^G$.
\end{proposition}

Note that the radius bound can become a diameter bound of $4(h-1)\delta+2$.
We will only need the case where $\delta=1$.

We can turn this result for unweighted graphs into a dilation for the vertex-weighted metrics as follows.

\begin{proposition}
\label{prop:weighted-rig-an}
Let\/ $h\geq 2$, let\/ $G=(V,E)$ be a region intersection graph over another graph\/ $G_0$, and assume\/ $G_0$ excludes\/ $K_h$ as a minor.
Then, for every assignment of vertex weights\/ $w\colon V\to\RR_{\geq 0}$, the metrics\/ $d_w^G$ and\/ $\dbar_w^G$ each have Assouad-Nagata dimension at most\/ $2^{h-2}-1$, via the respective\/ $(2^{h-2}-1)$-dimensional control functions\/ $\rho\mapsto(12h-6)\rho$ and\/ $\rho\mapsto(16h-8)\rho$.
\end{proposition}

\begin{proof}
Fix $\rho>0$.
For each $x\in V$, let $R_x$ denote the corresponding connected set of vertices (``region'') in $G_0$.
Whenever we take the union of the regions corresponding to a ball in either $d_w^G$ or $\dbar_w^G$, the resulting set is connected in $G_0$, because every such ball induces a connected subgraph of $G$.
We will use this fact to ``grow'' regions, similar to the proof of \cref{prop:an-to-sparse}:

\begin{claim}
Set\/ $\tau:=3$ and\/ $\overline\tau:=4$.

\smallskip

\noindent
\emph{Endpoint-light case:}
Let\/ $L:=\{v\in V\colon w(v)>2\rho\}$ and\/ $S:=V\setminus L$.
For each\/ $x\in S$, let\/ $R_x^\rho:=\bigcup_{y\colon d_w^G(x,y)\leq\rho/2}R_y$, and let\/ $G^\rho$ be the region intersection graph of\/ $\{R_x^\rho:x\in S\}$ over\/ $G_0$.
Then, for all distinct\/ $x,y\in S$, if\/ $d_w^G(x,y)\leq\rho$, then\/ $xy\in E(G^\rho)$, and if\/ $xy\in E(G^\rho)$, then\/ $d_w^G(x,y)\leq\tau\rho$.

\smallskip

\noindent
\emph{Endpoint-heavy case:}
Let\/ $L:=\{v\in V\colon w(v)>\rho\}$ and\/ $S:=V\setminus L$.
For each\/ $x\in S$, let\/ $\Rbar_x^\rho:=\bigcup_{y\colon\dbar_w^G(x,y)\leq\rho}R_y$, and let\/ $\Gbar^\rho$ be the region intersection graph of\/ $\{\Rbar_x^\rho\colon x\in S\}$ over\/ $G_0$.
Then, for all distinct\/ $x,y\in S$, if\/ $\dbar_w^G(x,y)\leq\rho$, then\/ $xy\in E(\Gbar^\rho)$, and if\/ $xy\in E(\Gbar^\rho)$, then\/ $\dbar_w^G(x,y)\leq\overline\tau\rho$.
\end{claim}

Assume the claim for the moment.

We first handle the endpoint-light metric.
Let $L$, $S$, and $G^\rho$ be as in the endpoint-light case of the claim.
By \cref{prop:rig-control}, the graph $G^\rho$ admits a $(2^{h-2},1)$-disjoint partition into sets of radius at most $2h-1$ in $d^{G^\rho}$.
By the claim, each subfamily in this partition is $\rho$-disjoint in $d_w^G$ (under the direct correspondence of $V(G^\rho)$ and $S$).
Also, every part has diameter at most $4h-2$ in $d^{G^\rho}$, and every edge of $G^\rho$ spans $d_w^G$-distance at most $\tau\rho$, so every part has diameter at most $(4h-2)\tau\rho=(12h-6)\rho$ in $d_w^G$.
Finally, add every vertex of $L$ as a singleton to any one fixed subfamily in the disjoint partition.
This subfamily remains $\rho$-disjoint, since every vertex in $L$ has $d_w^G$-distance greater than $\rho$ from every other vertex of $G$.
Therefore, $\rho\mapsto(12h-6)\rho$ is a $(2^{h-2}-1)$-dimensional control function for $d_w^G$.

We next handle the endpoint-heavy metric with a symmetric argument.
Let $L$, $S$, and $\Gbar^\rho$ be as in the endpoint-heavy case of the claim.
By \cref{prop:rig-control}, the graph $\Gbar^\rho$ admits a $(2^{h-2},1)$-disjoint partition into sets of radius at most $2h-1$ in $d^{\Gbar^\rho}$.
By the claim, each subfamily in this partition is $\rho$-disjoint in $\dbar_w^G$ (under the direct correspondence of $V(\Gbar^\rho)$ and $S$).
Also, every part has diameter at most $4h-2$ in $d^{\Gbar^\rho}$, and every edge of $\Gbar^\rho$ spans $\dbar_w^G$-distance at most $\overline\tau\rho$, so every part has diameter at most $(4h-2)\overline\tau\rho=(16h-8)\rho$ in $\dbar_w^G$.
Finally, add every vertex of $L$ as a singleton to any one fixed subfamily in the disjoint partition.
This subfamily remains $\rho$-disjoint, since every vertex in $L$ has $\dbar_w^G$-distance greater than $\rho$ from every other vertex of $G$.
Therefore, $\rho\mapsto(16h-8)\rho$ is a $(2^{h-2}-1)$-dimensional control function for $\dbar_w^G$.

It remains to prove the claim.

For the endpoint-light case, suppose first that $d_w^G(x,y)\leq\rho$ for distinct $x,y\in S$.
Let $P=v_0,\ldots,v_t$ be a shortest-path from $x=v_0$ to $y=v_t$ in the edge-weighted graph defining $d_w^G$, and let $i$ be maximal such that $d_w^G(x,v_i)\leq\rho/2$.
If $i=t$, then $y$ contributes to $R_x^\rho$, so $R_y\subseteq R_x^\rho\cap R_y^\rho$.
Otherwise, the suffix of $P$ from $v_{i+1}$ to $y$ has length less than $\rho/2$, so $v_i$ contributes to $R_x^\rho$ and $v_{i+1}$ contributes to $R_y^\rho$.
Since $v_iv_{i+1}\in E$, we have $R_{v_i}\cap R_{v_{i+1}}\neq\emptyset$, and therefore $R_x^\rho\cap R_y^\rho\neq\emptyset$.
Therefore, $xy\in E(G^\rho)$.

Now, suppose that $xy\in E(G^\rho)$.
Choose $a,b\in S$ with $d_w^G(x,a)\leq\rho/2$, $d_w^G(y,b)\leq\rho/2$, and $R_a\cap R_b\neq\emptyset$.
Then either $a=b$ or $ab\in E$.
Since $a,b\in S$, we have $d_w^G(a,b)\leq\frac{w(a)+w(b)}2\leq2\rho$.
Hence, $d_w^G(x,y)\leq d_w^G(x,a)+d_w^G(a,b)+d_w^G(b,y)\leq 3\rho=\tau\rho$.

For the endpoint-heavy case, suppose first that $\dbar_w^G(x,y)\leq\rho$ for distinct $x,y\in S$.
Then $y$ contributes to $\Rbar_x^\rho$, so $R_y\subseteq\Rbar_x^\rho\cap\Rbar_y^\rho$, and hence $xy\in E(\Gbar^\rho)$.

Now suppose that $xy\in E(\Gbar^\rho)$.
Choose $a,b\in S$ with $\dbar_w^G(x,a)\leq\rho$, $\dbar_w^G(y,b)\leq\rho$, and $R_a\cap R_b\neq\emptyset$.
Then either $a=b$ or $ab\in E$.
Since $a,b\in S$, we have $\dbar_w^G(a,b)\leq w(a)+w(b)\leq 2\rho$.
Hence $\dbar_w^G(x,y)\leq\dbar_w^G(x,a)+\dbar_w^G(a,b)+\dbar_w^G(b,y)\leq 4\rho=\overline\tau\rho$.
\end{proof}

\begin{theorem}[{Extended restatement of \cref{thm:pad-rig}}]
Let\/ $G_0$ be a\/ $K_h$-minor-free graph, $G=(V,E)$ be a region intersection graph over\/ $G_0$, and\/ $w\colon V\to\RR_{\geq 0}$ be an assignment of vertex weights.
Then the endpoint-light vertex-weighted shortest-path metric\/ $d_w^G$ admits a\/ $\bigl(16(12h-5)\cdot\break(1+(h-2)\ln 2),\frac{1}{8(12h-5)}\bigr)$-padded decomposition scheme, and the endpoint-heavy vertex-weighted shortest-path metric\/ $\dbar_w^G$ admits a\/ $\bigl(16(16h-7)(1+(h-2)\ln 2),\frac{1}{8(16h-7)}\bigr)$-padded decomposition scheme.
In particular, every string graph admits an\/ $(\Oh(1),\Omega(1))$-padded decomposition scheme.
\end{theorem}

\begin{proof}
Assume $h\geq 2$, since otherwise $G$ is empty.
By \cref{prop:weighted-rig-an}, the metric $d_w^G$ has a $(2^{h-2}-1)$-dimensional control function $\rho\mapsto(12h-6)\rho$, while $\dbar_w^G$ has a $(2^{h-2}-1)$-dimensional control function $\rho\mapsto(16h-8)\rho$.
By \cref{prop:an-to-sparse}, this gives a $(24h-10,2^{h-2})$-sparse cover scheme for $d_w^G$ and a $(32h-14,2^{h-2})$-sparse cover scheme for $\dbar_w^G$.
By \cref{prop:sparse-to-pad}, $d_w^G$ admits a padded decomposition scheme with parameters $\bigl(8(24h-10)\bigl(1+\ln(2^{h-2})\bigr),\frac{1}{4(24h-10)}\bigr)=\bigl(16(12h-5)\bigl(1+(h-2)\ln 2\bigr),\frac{1}{8(12h-5)}\bigr)$, while $\dbar_w^G$ admits a padded decomposition scheme with parameters $\bigl(8(32h-14)\bigl(1+\ln(2^{h-2})\bigr),\frac{1}{4(32h-14)}\bigr)=\bigl(16(16h-7)\bigl(1+(h-2)\ln 2\bigr),\frac{1}{8(16h-7)}\bigr)$.
\end{proof}

Finally, we note that this result need not be made algorithmic if constant factors are unimportant, since Krauthgamer and Roughgarden~\cite{krauthgamer2011metric} devised a polynomial-time bi-criteria approximation algorithm for padded decompositions.

\section{Vertex Expansion, Metric Spread, Balanced Separators, and Relaxations}\label{sec:vertexexpansion}

In this section, we will prove \cref{prop:sp-to-embed} and \cref{prop:embed-to-sep}.

Fix a graph $G=(V,E)$ with $n$ vertices.
For a subset $S\subseteq V$ of the vertices, let $\dotpartial S$ be the minimum-cardinality vertex cover of $E(S,V\setminus S)$ (ties broken arbitrarily), which we call the \defn{cover-boundary}.
We define the \defn{covered-vertex expansion} of $S$ in $G$ as $\psi(S):=\smash[b]{\tfrac{|\dotpartial S|}{|S|}}$%
\footnote{These definitions of vertex expansion and boundary differ slightly from the usual ones, but are easier to work with for our purposes.},
and the cover-vertex expansion of $G$ itself is $\psi(G):=\min_{S\subseteq V\colon|S|\leq\frac{n}{2}}\frac{|\dotpartial S|}{|S|}$.
We will use cover-vertex expansion as an intermediary between $\Qbar(G)$ and vertex separators.

\begin{observation}
\label{obs:psi-to-sep}
Let\/ $\alpha>0$.
If, for every induced subgraph\/ $H$ of\/ $G$ with\/ $n'$ vertices, $H$ has vertex expansion at most\/ $\frac{\alpha}{n'}$, then\/ $G$ admits a\/ $\frac{2}{3}$-balanced vertex separator of size at most\/ $\alpha$.

Moreover, if, for every induced subgraph\/ $H$ of\/ $G$ with\/ $n'$ vertices, a set\/ $S$ with\/ $|S|\leq\frac{n'}{2}$ whose vertex expansion is at most\/ $\frac{\alpha}{n'}$ can be found in\/ $f(n')$ time (for some monotone function\/~$f$), then \/ $\frac{2}{3}$-balanced vertex separator of\/ $G$ with size at most\/ $\alpha$ can be found in\/ $\Oh(n\cdot f(n))$ time.
\end{observation}

\begin{proof}
The construction follows a simple iterative algorithm.
Start with $G_0:=G$.
For each iteration $i$, with $V_i$ as the vertex set of $G_i$, let $n_i:=|V_i|$, find some set $S_i$ satisfying $\frac{|\dotpartial S_i|}{|S_i|}\leq\frac{\alpha}{n_i}$ and $|S_i|\leq\frac{n_i}{2}$, and then let $G_{i+1}:=G_i[V_i\setminus S_i\setminus\dotpartial S_i]$.
Halt when $|V_i|\leq\frac{2n}{3}$.
Let $t$ be the index where this occurs, so $|V_t|\leq\frac{2n}{3}$.

Let $X:=\cup_{0\leq i<t}\;\dotpartial S_i$.
This is a $\frac{2}{3}$-balanced vertex separator by definition, and its construction satisfies the time complexity requirements, so it only remains to check its size.

For each iteration $0\leq i<t$, we have $\frac{|\dotpartial S_i|}{|S_i|}\leq\frac{\alpha}{n_i}$, and hence $|\dotpartial S_i|\leq|S_i|\frac{\alpha}{n_i}$.
Thus, $|X|=\sum_{0\leq i<t}|\dotpartial S_i|\leq\alpha\sum_{0\leq i<t}\smash[t]{\frac{|S_i|}{n_i}}$.
Then, $n_{i+1}=n_i-|S_i\cup\dotpartial S_i|\leq n_i-|S_i|$, so rearranging gives $|S_i|\leq n_i-n_{i+1}$.
Combining these, we get
\[|X| \leq \alpha\left(\frac{|S_{t-1}|}{n_{t-1}}+\sum_{0\leq i<t-1}\frac{n_i-n_{i+1}}{n_i}\right){.}\]

Note that, for any $q>0$, we have $1-q\leq\ln\smash[b]{\frac{1}{q}}$.
Hence, for any $a,b>0$, we have $1-\frac{b}{a}=\frac{a-b}{a}\leq\ln\frac{a}{b}$.
Using $a=n_i$ and $b=n_{i+1}$, we get $\sum_{0\leq i<t-1}\frac{n_i-n_{i+1}}{n_i}\leq\sum_{0\leq i<t-1}\ln\frac{n_i}{n_{i+1}}=\ln\frac{n_0}{n_{t-1}}$.

For $i<t-1$, we have $n_{i+1}>\frac{2}{3}n$ and $n_i\leq n$, so $\frac{n_0}{n_{t-1}}\leq\frac{3}{2}$.
We also have $|S_{t-1}|\leq\frac{n_{t-1}}{2}$.
Therefore, $|X|\leq\alpha\smash[t]{\bigl(\frac{1}{2}+\ln\frac{3}{2}\bigr)}<\alpha$.
\end{proof}

Next, we can relate extremal line-embedded spread to cover-vertex expansion.

\begin{lemma}
\label{lemma:embed-to-psi}
For a graph\/ $G$ with\/ $n$ vertices,
\[\frac{n}{\psi(G)} \leq \Qbar(G) \leq \frac{4n}{\psi(G)}{.}\]
\end{lemma}

\begin{proof}
We first show the left-hand-side inequality.
Let $S\subseteq V$ have $|S|\leq\smash[b]{\frac{n}{2}}$ and $\psi(S)=\psi(G)$.
Let $f(v):=\frac{1}{2|\dot\partial S|}$ if $v\in S$, and $f(v):=\frac{-1}{2|\dot\partial S|}$ otherwise.
Let $w(v):=\frac{1}{|\dot\partial S|}$ if $v\in\dotpartial S$, and $w(v):=0$ otherwise.
This is a feasible solution for $\Qbar(G)$.
To bound its objective value, observe that $\sum_{u,v\in V}|f(u)-f(v)|=2\sum_{u\in S}\sum_{v\notin S}\smash[b]{\frac{1}{|\dot\partial S|}}=\smash[b]{\frac{2|S||V\setminus S|}{|\dot\partial S|}}\geq\smash[b]{\frac{n|S|}{|\dot\partial S|}}=\smash[b]{\frac{n}{\psi(G)}}$, where the inequality uses $|V\setminus S|=n-|S|\geq\frac{n}{2}$.

We now show the right-hand-side inequality, by generalizing a threshold rounding argument of Lee \cite[Proof of Theorem~2.3]{lee2017separators}, which itself is a variation of a standard argument.

Let $(w,f)$ be the optimal and feasible solution to $\Qbar(G)$.
For each $\theta\in\RR$, define
\[A_\theta:=\{v\in V\colon f(v)\leq\theta\}{,} \qquad S_\theta:=\{v\in V\colon|f(v)-\theta|\leq w(v)\}{,} \qquad B_\theta:=\{v\in V\colon f(v)>\theta\}{.}\]
If $u\in A_\theta\setminus S_\theta$ and $v\in B_\theta\setminus S_\theta$, then $f(u)\leq\theta-w(u)$ and $f(v)>\theta+w(v)$, so $|f(u)-f(v)|>w(u)+w(v)$.
Hence, by feasibility, $uv\notin E$, and thus $E(A_\theta\setminus S_\theta,B_\theta\setminus S_\theta)=\emptyset$ for all $\theta\in\RR$.

Fix $\theta\in\RR$ with $A_\theta,B_\theta\neq\emptyset$.
Assume without loss of generality that $|A_\theta|\leq|B_\theta|$.
Note that $S_\theta$ is a vertex cover for $E(A_\theta,B_\theta)$, so $|\dotpartial(A_\theta)|\leq|S_\theta|$.
Therefore,
\[\psi(G) \leq \frac{|\dotpartial(A_\theta)|}{|A_\theta|} \leq \frac{|S_\theta|}{|A_\theta|} \leq \frac{|S_\theta|}{|A_\theta|}\cdot\frac n{|B_\theta|}{.}\]
Rearranging gives $|S_\theta|\geq\frac{\psi(G)}n|A_\theta||B_\theta|$.
This also trivially holds when one of $A_\theta,B_\theta$ is $\emptyset$.

Using this inequality,
\begin{align*}
\int_{\theta\in\RR}|S_\theta|\,d\theta &\geq \frac{\psi(G)}n\int_{\theta\in\RR}|A_\theta||B_\theta|\,d\theta = \frac{\psi(G)}n\int_{\theta\in\RR}\sum_{u,v\in V}\mathbf{1}_{\{f(u)\leq\theta<f(v)\}}\,d\theta \\
&= \frac{\psi(G)}n\sum_{u,v\in V}\int_{\theta\in\RR}\mathbf{1}_{\{f(u)\leq\theta<f(v)\}}\,d\theta = \frac{\psi(G)}n\sum_{u,v\in V}\max\{f(v)-f(u),0\} \\
&= \frac{\psi(G)}{2n}\sum_{u,v\in V}|f(u)-f(v)|{.}
\end{align*}
Moreover,
\[\int_{\theta\in\RR}|S_\theta|\,d\theta \leq \sum_{v\in V}\int_{\theta\in\RR}\mathbf{1}_{\{|f(v)-\theta|\leq w(v)\}}d\theta = \sum_{v\in V}2w(v) \leq 2{.}\]
Hence, combining the two inequalities, we get
\[\Qbar(G) = \sum_{u,v\in V}|f(u)-f(v)| \leq \frac{4n}{\psi(G)}{.}\qedhere\]
\end{proof}

Using this, we may prove \cref{prop:embed-to-sep}.

\restate{proposition}{prop:embed-to-sep}

\begin{proof}
By \cref{lemma:embed-to-psi}, every connected induced subgraph $H$ satisfies
$\psi(H)\leq\frac{4\beta}{n'}$.
If $H$ is disconnected, then $\psi(H)=0$.
Therefore, by \cref{obs:psi-to-sep}, $G$ has a separator of size at most $4\beta$.
\end{proof}

Next, we prove \cref{prop:sp-to-embed}.

\restate{proposition}{prop:sp-to-embed}

\begin{proof}
Let $w\colon V\to\RR_{\geq 0}$ be an optimal solution for $Q(G)$.
Let $f\colon V\to\RR$ be a non-expansive embedding of $(V,\dbar_w)$ into the line with average distortion at most $\alpha$.
The average distortion bound gives $Q(G)=\sum_{u,v\in V}\dbar_w(u,v)\leq\alpha\sum_{u,v\in V}|f(u)-f(v)|$, so it remains only to check that $(w,f)$ is a feasible solution for $\Qbar(G)$.

For every edge $uv\in E$, non-expansiveness gives
\[|f(u)-f(v)| \leq \dbar_w(u,v) \leq w(u)+w(v){,}\]
and $\sum_{x\in V}w(x)\leq 1$.
So $(w,f)$ is a feasible solution for $\Qbar(G)$.
\end{proof}

\section{Conclusion: Impact on Previous Works}\label{sec:conclusion}

As the main results of this paper, we proved \cref{thm:mainrig} and \cref{thm:mainim} that were previously stated, respectively, by~\cite{lee2017separators} and by Korhonen and Lokshtanov~\cite{korhonen2024induced}.
We also proved \cref{thm:pad-rig} that fixes the result ($\star$) claimed by Lee~\cite{lee2017separators}, but with weaker parameters.

In this section, we go over papers building on \cref{thm:mainrig}, \cref{thm:mainim}, and the result ($\star$) or their respective proofs, and we discuss whether or not (and how) they have been fixed as a result of our work.

\subsection*{Results Directly Using Separator Theorems}

A number of papers~\cite{bonnet2019optimality,okrasa2020subexponential,novotna2021subexponential,berthe2025pushing} used the separator theorems of Lee~\cite{lee2017separators} as black boxes, usually to design divide-and-conquer algorithms.
We remark that in typical cases, exchanging the separator of size $\Oh(\sqrt{m})$ to one of size $\Oh(\sqrt{m}\log m)$, given by Matoušek~\cite{matousek2014near} makes the running time only slightly worse.
Nevertheless, with \cref{thm:mainrig} and \cref{thm:mainim} reproven, the results claimed in these papers are now fully fixed.
Similarly, de~Berg, Jansen, and Lamme~\cite{deberg2025star} did not use the results of Lee since they were already aware of the mistake, and thus they used the weaker $\Oh(\sqrt{m}\log m)$ separator.
With \cref{thm:mainrig} and \cref{thm:mainim} reproven, the results claimed in this paper can now be slightly improved.

\subsection*{Results Directly Using Padded Decompositions}

The main result of the paper \emph{A Constant Factor Approximation for Navigating Through Connected Obstacles in the Plane} by Kumar, Lokshtanov, Saurabh, and Suri~\cite{kumar2021constant} directly relied on ($\star$), but was needed only for string graphs.
Thus, the result is now fixed by \cref{thm:pad-rig}.

Similarly, in the paper \emph{Point Separation and Obstacle Removal by Finding and Hitting Odd Cycles} by Kumar, Lokshtanov, Saurabh, Suri, and Xue~\cite{kumar2022point}, one of the main results for the ``obstacle removal'' problem directly relied on ($\star$) for string graphs and is wholly fixed by \cref{thm:pad-rig}.

Now let us go back to the paper that started it all, i.e., \emph{Separators in Region Intersection Graphs} by Lee~\cite{lee2017separators}.
As an application of the padding decomposition scheme, Lee showed a result concerning eigenvalue bounds for region intersection graphs.
For a graph $G$ with degree-vector $\vec d$ and adjacency matrix $A$, its Laplacian matrix is $L:=\operatorname{diag}(\vec d)-A$, and the eigenvalues of $L$ are denoted $\lambda_1(G)\leq\lambda_2(G)\leq\cdots\leq\lambda_n(G)$.
Lee~\cite[Theorem~1.11]{lee2017separators} claimed that for any graph $G$ with maximum degree $\Delta$ that is a region intersection graph over another graph $G_0$, where $G_0$ is $K_h$-minor-free, $\lambda_k(G)=\Oh(\Delta^2 h^6\log h)\frac{k}{|V_G|}$.
The proof applied ($\star$) in full generality, and so we recover a weaker bound via \cref{thm:pad-rig}.
The resulting bound is the same up to factors depending on $h$.
In particular, if $h$ is a constant, then the bound is $\lambda_k(G)=\Oh(\Delta^2)\frac{k}{|V_G|}$.
Lee applied the methods of Kelner, Lee, Price, and Teng~\cite{kelner2011metric} for this proof, and it seems quite likely that these results could be derived using only the weaker notion of Rabinovich embeddings, but we leave this to future work.

\subsection*{Results Using Rabinovich Embeddings}

A closely related problem was considered in the paper \emph{Reweighted Spectral Partitioning Works: A Simple Algorithm for Vertex Separators in Special Graph Classes} by Spalding-Jamieson~\cite{spaldingjamieson2025reweighted}.
Again, let $G$ be a region intersection graph over another graph $G_0$, where $G_0$ is $K_h$-minor-free, and assume $G$ has maximum degree at most $\Delta$.
One of the results claimed in the earlier preprints of Spalding-Jamieson's work is that, for the so-called maximally reweighted second eigenvalue of $G$, denoted $\lambda_2^*(G)$, it holds that $\lambda_2^*(G)=\Oh(\Delta h^2\log h)\frac{1}{|V_G|}$.
This would be a generalization of the earlier mentioned bound on $\lambda_k$ by Lee~\cite{lee2017separators} for the special case where $k=2$.
However, since the proof of this bound also relied on ($\star$), it was removed from the later versions of the paper.
This step can now be fixed by directly relying on Rabinovich embeddings instead of padded decompositions.
Hence, the missing step for the result follows from \cref{thm:embed}.

\subsection*{Results Using Klein-Plotkin-Rao Decomposition}

Finally, let us discuss the so-called \emph{Klein-Plotkin-Rao (KPR) decomposition}~\cite{klein1993excluded} (sometimes called a low-diameter decomposition).
Originally shown for $K_h$-minor-free graphs, this result says that we can remove a small fraction of vertices so that any two vertices in the same component of the resulting graph are close to each other in the original graph.
Lee~\cite{lee2017separators} claimed that an analogous result holds for string graphs, and, more generally, graphs excluding a fixed induced minor.

The KPR decomposition for string graphs and induced-minor-free graphs was later used in the paper \emph{A\/ $1.9999$-Approximation Algorithm for Vertex Cover on String Graphs} by Lokshtanov, Panolan, Saurabh, Xue, Zehavi~\cite{lokshtanov2024approximation}, as well as the (first version of) the preprint \emph{Treewidth is Polynomial in Maximum Degree on Graphs Excluding a Planar Induced Minor} by Bonnet, Hodor, Korhonen, and Masařík~\cite{bonnet2023treewidth}.

For the string graph results in the first paper, the required theorems \cite[Theorems 5 and~9]{lokshtanov2024approximation} follow from applying our \cref{thm:pad-rig} with unweighted shortest-path metrics using essentially-unchanged proofs.
Note that the string representation is not required since we may use an approximation algorithm of Krauthgamer and Roughgarden~\cite{krauthgamer2011metric}.
Their results also extend to region intersection graphs over $K_h$-minor-free graphs in this manner, but extending them to induced-minor-free graphs via an analogous low-diameter decomposition result remains an open problem.

\section*{Acknowledgements}

We are grateful to anonymous reviewers for pointing out some issues in the first version of the paper, and especially for carefully checking the formalization in Lean.

The results in the main body of the paper have been obtained entirely by the human authors, with no AI assistance.
AI tools were used for the Lean formalization (Aristotle and GPT 5.6 Sol), for obtaining an initial form of the improvement in \cref{app:improved-bound} (GPT 5.6 Sol), and for proofreading the paper (Claude Opus 5.5).

\bibliographystyle{alphaurl}
\bibliography{ref}

\appendix
\crefalias{section}{appendixsection}

\section{Planar Embeddings of Vertex-Weighted String Graphs}
\label{app:weighted-planar}

We first state the result of Davies that allows us to embed an edge-weighted planar graph in an unweighted one, approximately keeping the distances.

\begin{theorem}[{Davies~\cite[Theorem~1.5]{davies2025string}}]
\label{thm:planar-unweighting}
Let\/ $P$ be a countable planar graph whose edge lengths belong to\/ $(0,1]$, and let\/ $d_P$ be its weighted shortest-path metric.
There is an unweighted planar graph\/ $H$ on the same vertex set such that
\[\frac{d_P(u,v)}{47321600}\leq d_H(u,v)\leq 162d_P(u,v)+1\qquad\text{for all }u,v\in V(P){,}\]
where\/ $d_H$ is the unweighted shortest-path metric of\/ $H$.
\end{theorem}

We now show that no universal distortion bound is possible for embeddings of vertex-weighted string graphs into edge-weighted planar graphs.

\begin{proposition}
For every\/ $D\geq 1$, there is a vertex-weighted string graph whose endpoint-light shortest-path metric admits no embedding with distortion at most\/ $D$ into any edge-weighted planar graph.
\end{proposition}

\begin{proof}
Suppose, for a contradiction, that such embeddings exist for a universal constant $D$.
Let $Q$ be the $N\times N$ grid, and obtain $G$ from $Q$ by adding a universal vertex $z$.
Give every grid vertex weight $1$ and give $z$ weight $R$, where $R\geq 1$ and $N$ will be chosen later.
It is straightforward to verify that $G$ is a string graph.
For the endpoint-light metric $d:=d_w^G$, we have
\[d(x,y)=\min\{d_Q(x,y),R+1\}\quad(x,y\in V(Q)){,}\qquad d(z,x)=\frac{R+1}{2}\quad(x\in V(Q)){.}\]
In particular, $G$ has (metric) diameter at most $R+1$, while distances smaller than $R+1$ between grid vertices agree with their grid distances.

By assumption, after rescaling the target edge lengths, there is an embedding into an edge-weighted planar graph $P$ with
\[d(x,y)\leq d_P(\phi(x),\phi(y))\leq Dd(x,y){.}\]
Contract zero-length edges and subdivide edges so that all edge lengths belong to $(0,1]$.
By \cref{thm:planar-unweighting}, we may replace $P$ by an unweighted planar graph $H$ so that, for constants $a,b\geq 1$ depending only on $D$,
\[a^{-1}d(x,y)\leq d_H(\phi(x),\phi(y))\leq b\,d(x,y){.}\]
Here the additive constant in that theorem is absorbed into $b$, since distinct vertices of $G$ have distance at least $1$.
Set $L:=b(R+1)$ and restrict $H$ to the ball of radius $2L$ about $\phi(z)$.
This ball contains every shortest path between image vertices, so their distances are unchanged, and the resulting graph has radius at most $2L$.

Assign each vertex $u$ of $H$ mass $\mu(u):=|\phi^{-1}(u)\cap V(Q)|$, of total mass $N^2$.
The bounded-radius separator lemma of Lipton and Tarjan~\cite[Lemma~2]{lipton1979separator} gives a set $S_H$ of size $\Oh(L)=\Oh(R)$ such that every component of $H-S_H$ has mass at most $2N^2/3$.
All constants hidden in this proof may depend on $D$, but not on $R$ or $N$.
Let $r:=\lceil b\rceil$ and define
\[S_Q:=\{x\in V(Q):d_H(\phi(x),S_H)\leq r\}{.}\]
Every grid edge has $d$-length $1$, so its endpoints have images joined by a path in $H$ of length at most $b\leq r$.
If both endpoints lie outside $S_Q$, such a path avoids $S_H$.
Consequently, the image of each component of $Q-S_Q$ lies in one component of $H-S_H$, and $S_Q$ is a $\frac{2}{3}$-balanced separator of $Q$.
Adding $z$ also gives a balanced separator of $G$.

For each $u\in S_H$, the set
\[X_u:=\{x\in V(Q):d_H(\phi(x),u)\leq r\}\]
has $d$-diameter at most $2ar$.
Choose $R$ so that $R+1>2ar$.
Then $X_u$ has grid diameter at most $2ar$ and thus contains at most $(2\lfloor 2ar\rfloor+1)^2=\Oh(1)$ vertices.
Since $S_Q=\bigcup_{u\in S_H}X_u$, we obtain $|S_Q|=\Oh(R)$.
However, every $\frac{2}{3}$-balanced separator of the $N\times N$ grid has size at least $N/3$: deleting fewer vertices leaves more than $2N/3$ intact rows and an intact column, whose union is connected and contains more than $2N^2/3$ vertices.
Taking $N$ sufficiently large compared with $R$ gives a contradiction.
\end{proof}

Since $d_w^G\leq\dbar_w^G\leq 2d_w^G$, the same conclusion holds for endpoint-heavy vertex-weighted shortest-path metrics.

\begin{corollary}
For every\/ $D\geq 1$, there is a vertex-weighted string graph whose endpoint-heavy shortest-path metric admits no embedding with distortion at most\/ $D$ into any edge-weighted planar graph.
\end{corollary}

\section{An Improved Embedding Bound}
\label{app:improved-bound}

In this short appendix section, we improve the dependence on $h$ in \cref{thm:embed}, and consequently \cref{thm:mainim}.
Specifically, we obtain the following pair of results.

\begin{theorem}
\label{thm:improved-embed}
Let\/ $G=(V,E)$ be a connected graph excluding\/ $\Ksubdiv{h}$ as an induced minor.
Then, for any assignment of vertex weights\/ $w\colon V\to\RR_{\geq 0}$, the metrics\/ $d_w^G$ and\/ $\dbar_w^G$ admit non-expansive embeddings into the line with average distortion\/ $\Oh(\log h)$.

Furthermore, there is a polynomial-time in\/ $|V|$ algorithm that, given a graph $G$ and an assignment of vertex weights\/ $w\colon V\to\RR_{\geq 0}$, obtains either an induced minor model of\/ $\Ksubdiv{h}$ in\/ $G$, or such a non-expansive embedding into the line.
\end{theorem}

We note that \cref{thm:improved-embed} has the best-possible dependence on $h$ (realized by expanders).

\restate{theorem}{thm:improved-main}

Note that \cref{thm:improved-main} follows immediately from \cref{thm:improved-embed} and the standard steps in \cref{sec:standard-relax-embed-sep-results}.

The initial form of the improvement in this appendix was devised by GPT 5.6 Sol.
It has since undergone significant simplification by the authors in its presentation here.
A formalization of this result appears in the same repository alongside the base results.

The intuition for proving \cref{thm:improved-embed} is as follows.
First, note that \cref{thm:pure-pairs-subdiv} can be applied with ``terminals''.
That is, for a subset of vertices, $U\subset V$, we can use an outer measure $\mu(X):=|X\cap U|/|U|$.
With such an outer measure, we maintain the polynomial-time computability stated in \cref{thm:pure-pairs-subdiv-poly}.
Then, the existing proof of \cref{thm:embed} in \cref{sec:main-proofs} aims to directly map the applicable cases of \cref{thm:pure-pairs-subdiv-poly} to one of \cref{lemma:large-core} and \cref{lemma:well-spread}.
In contrast, in this section we will apply \cref{lemma:large-core} first to a trivial case, and then we will iteratively apply \cref{thm:pure-pairs-subdiv-poly} to a shrinking terminal set of the vertices, via a metric peeling argument.
After enough iterations, we eventually apply \cref{lemma:well-spread} to the residual terminal set and all the peeled vertices.

\begin{proof}[Proof of \cref{thm:improved-embed}]
It suffices to consider only the metric $d:=d_w^G$, since $d\leq\dbar_w^G\leq 2d$.

Write $n:=|V|$ and $\Delta:=n^{-2}\sum_{x,y\in V}d(x,y)$.
As in the proof of \cref{thm:embed}, we may assume $h>1$.
If $\Delta=0$, the theorem is trivial, so we also assume $\Delta>0$.
Set
\[\epsilon:=\frac{13}{40h}{,} \qquad \eta:=\frac{1}{32h^2}{,} \qquad \delta:=\frac{\Delta}{33}{.}\]
Note that $\epsilon\geq\eta$ and
\[(2h+2)\epsilon+\left(h+\binom{h}{2}\right)\eta=\frac{213(h+1)}{320h}<1{.}\]

We first eliminate a trivial case.
If some ball $B(x,12\delta)$ (for $x\in V$) contains at least $n/4$ vertices, \cref{lemma:large-core} (with $\kappa:=1/4$ and $c=33/12>2$) gives a non-expansive embedding with constant average distortion.
Henceforth, assume without loss of generality that every such ball contains fewer than $n/4$ vertices.

Set
\[k:=\left\lceil\log_2\frac{1}{\eta}\right\rceil=\Oh(\log h){,}\qquad r:=\frac{\delta}{2k}=\frac{\Delta}{66k}{.}\]
Initialize $A_0:=\emptyset$ and $B_0:=V$.
We will perform a sequence of iterations.
In each iteration $i$, we will move some vertices from $B_i$ into $A_i$, while discarding others, obtaining $A_{i+1}\supseteq A_i$ and $B_{i+1}\subseteq B_i$.
We will halt at the first index $t$ with $|B_t|\leq n/2$, or upon finding an induced minor model of $\Ksubdiv{h}$.
For each iteration $i$ with $|B_i|>n/2$, we perform the following three steps.

\begin{enumerate}
\item We seek a set $S\subseteq B_i$ where
\[|S|>\eta|B_i| \quad \text{and} \quad |B_i\cap B(S,\delta/2)|\leq |B_i|/2{.}\]
Apply \cref{thm:pure-pairs-subdiv} to $G$ with target $K_h$ and the outer measure $\mu(X):=|X\cap B_i|/|B_i|$, which satisfies $\mu(V)=1$.
We handle each of its cases separately.
\begin{enumerate}
\item[(a)] Assume there is $x\in V$ with $\mu(B(x,11\delta))>\eta$.
Set $S:=B_i\cap B(x,11\delta)$, so $|S|>\eta|B_i|$.
By the triangle inequality,
\[B(S,\delta/2)\subseteq B(x,11\delta+\delta/2)\subseteq B(x,12\delta){.}\]
Thus $|B_i\cap B(S,\delta/2)|<n/4<|B_i|/2$, as required.
\item[(b)] Assume there are sets $P,Q\subseteq V$ with $\mu(P),\mu(Q)>\epsilon$ and $d(P,Q)>\delta$.
The balls $B(P,\delta/2)$ and $B(Q,\delta/2)$ are disjoint by the triangle inequality.
Hence, after interchanging $P$ and $Q$ if necessary, we may assume $|B_i\cap B(P,\delta/2)|\leq|B_i|/2$.
Set $S:=B_i\cap P$.
Then $|S|>\epsilon|B_i|\geq\eta|B_i|$ and $B(S,\delta/2)\subseteq B(P,\delta/2)$, which gives both required bounds.
\item[(c)] Assume $G$ contains $\Ksubdiv{h}$ as an induced minor.
We obtain its induced minor model and halt.
\end{enumerate}
\item
We seek two nested neighborhoods of $S$ in $B_i$ whose radii differ by $r$ and whose sizes differ by at most a factor of two.
For $j=0,\ldots,k$, let $T_j:=\{x\in B_i\colon d(x,S)\leq jr\}$.
Since $kr=\delta/2$,
\[S\subseteq T_0\subseteq T_1\subseteq\cdots\subseteq T_k=B_i\cap B(S,\delta/2){.}\]
Choose an index $j<k$ with $|T_{j+1}|\leq2|T_j|$.
Such an index $j$ exists, since otherwise
\[|B_i|/2\geq |T_k|>2^k|T_0|>2^k\eta|B_i|\geq|B_i|{.}\]
\item Set
\[A_{i+1}:=A_i\cup T_j{,}\qquad B_{i+1}:=B_i\setminus T_{j+1}{.}\]
\end{enumerate}

Each iteration removes between $1$ and $|B_i|/2$ vertices from $B_i$, since $\emptyset\neq S\subseteq T_{j+1}\subseteq T_k$.
Thus, unless we find a model of $\Ksubdiv{h}$, the process stops after at most $n$ iterations with
\[n/4<|B_{t-1}|/2\leq|B_t|\leq n/2{.}\]
The sets added to $A_i$ are disjoint, and each has size $|T_j|\geq |T_{j+1}|/2$.
Therefore,
\[|A_t|\geq\frac{1}{2}\sum_{i=0}^{t-1}(|B_i|-|B_{i+1}|)=\frac{n-|B_t|}{2}\geq\frac{n}{4}{.}\]

We next show that $d(A_t,B_t)>r$.
In iteration $i$, for $x\in T_j$ and $y\in B_{i+1}$, the triangle inequality gives
\[d(x,y)\geq d(y,S)-d(x,S)>(j+1)r-jr=r{.}\]
Since $B_t\subseteq B_{i+1}$, every vertex added to $A_t$ is more than $r$ from every vertex of $B_t$.

By \cref{lemma:well-spread} (with $\kappa:=1/4$ and $c=66k$), the non-expansive map $f(x):=d(x,A_t)$ on $V$ has average distortion at most $1056k=\Oh(\log h)$.
Each of the at most $n$ iterations runs in polynomial-time by the algorithm underlying \cref{thm:pure-pairs-subdiv} since the terminal measure can be computed in polynomial time.
\end{proof}

\end{document}